\documentclass[reqno,11pt]{amsart}
\usepackage{tikz,pgfplots, subcaption, graphicx, amssymb}
\usetikzlibrary{snakes,arrows,patterns}
\newtheorem{theorem}{Theorem}[section]
\newtheorem{corollary}[theorem]{Corollary}

\theoremstyle{remark}
\newtheorem{remark}[theorem]{Remark}

\numberwithin{equation}{section}
\allowdisplaybreaks

\newcommand{\defn}[1]{{\emph{#1}}} 

\newcommand{\C}{\mathbb C}

\newcommand{\Z}{\mathbb Z}

\newcommand{\eqdef}{\mbox{\,\raisebox{0.2ex}{\scriptsize\ensuremath{\mathrm:}}\ensuremath{=}\,}} 

\DeclareMathOperator{\comp}{comp}

\newcommand{\eBinom}[3]{{\begin{bmatrix}#1\\#2\end{bmatrix}}_{#3}} 
\newcommand{\qBinom}[2]{{\begin{bmatrix}#1\\#2\end{bmatrix}}_{q}} 

\newcommand{\elliptic}[1]{[#1]_{a,b;q,p}}	
\definecolor{dblackcolor}{rgb}{0.0,0.0,0.0}
\definecolor{dredcolor}{rgb}{0.9,0.3,0.4}
\definecolor{dbluecolor}{rgb}{0.01,0.02,0.7}
\definecolor{dgreencolor}{rgb}{0.2,0.5,0.0}
\definecolor{dgraycolor}{rgb}{0.30,0.3,0.30}
\definecolor{gr}{rgb}{0.10,0.5,0.20}
\definecolor{myblue}{HTML}{0489B1}
\definecolor{myred}{HTML}{A40000}

\usepackage{hyperref}
\hypersetup{colorlinks=true, citecolor=black, linkcolor=black}

\newcommand{\da}{\hspace{-2pt}\downarrow}
\newcommand{\ua}{\hspace{-2pt}\uparrow}
\newcommand{\va}{\textbf{a}}

\newcommand{\vx}{\textbf{x}}
\newcommand{\fall}[3]{#1 \da_{#3}^{#2}}

\newcommand{\fac}[2]{{#1}_{#2}!}

\usepackage{todonotes}

\newcommand{\cf}{\textit{cf.}~} 

\author[Josef K\"ustner]{Josef K\"ustner$^{*}$}
\address{Fakult\"at f\"ur Mathematik, Universit\"at Wien,
	Oskar-Morgenstern-Platz~1, A-1090 Vienna, Austria}
\email{josef.kuestner@univie.ac.at}
\thanks{$^{*}$ The author was partially supported by Austrian Science Fund FWF
10.55776/P32305.}

      \title[Explicit formulae]{Explicit formulae for generalized Stirling and Eulerian numbers}


\begin{document}

\begin{abstract}
In this article we generalize the $q$-difference operator due to Carlitz in order to derive explicit sum formulae for several extensions of Stirling numbers of the second kind, including complete homogeneous symmetric functions, complementary symmetric functions, $r$-Whitney numbers and elliptic analogues of rook, Stirling and Lah numbers. 
Furthermore, we generalize Carlitz' $q$-Eulerian numbers to a Lagrange polynomial extension. We define them by generalizing Worpitzky's identity appropriately, and derive a recursion and an explicit sum formulae. Special cases include $r$-Whitney Eulerian numbers and elliptic Eulerian numbers.
\end{abstract}

\maketitle

\section{Introduction}\label{sec:stirling_intro}
The classical \defn{Stirling numbers of the second kind}, $S(n,k)$, count set partitions of an $n$-elementary set into $k$ non-empty sets. They fulfil the recurrence relation
\begin{subequations} \label{eqn:stirling_rec}
	\begin{align}
		S(n,k) &= 0, \qquad \text{for $k<0$ or $k>n$}, \\
		S(0,0) &= 1, \\
		\intertext{and, for $k \geq 0$,}
		S(n+1,k) &= S(n,k-1) + k S(n,k).
	\end{align}
\end{subequations}
Given a nonnegative integer $n$ and a variable $z$, we define the corresponding \defn{falling factorial} as
\[
z \da_n \eqdef z(z-1)\cdots (z-n+1)
\]
for $n>0$ and $z \da_0 = 1$. Then, $S(n,k)$ are the connection coefficients in
\begin{align}\label{eqn:stirling_con}
	z^n = \sum_{k=0}^{n} S(n,k) z \da_k.
\end{align}
Let $E_z$ be the shift operator that shifts a polynomial $p(z)$ by $E_zp(z) = p(z+1)$ and $\Delta$ be the difference operator $\Delta = E_z - I$. Then it is a classical result that we can derive an explicit formula for $S(n,k)$ from \eqref{eqn:stirling_con} by applying the $k$-th power of the difference operator $\Delta$ to both sides of the equation to obtain
\begin{align}
	\label{eqn:stirling_exp}
	S(n,k) = \frac{1}{k!} \sum_{j = 0}^{k} (-1)^j \binom{k}{j} (k-j)^n.
\end{align}
We recommend \cite{Cig1,Rota} for further reading on the theory of difference operators.

Following Carlitz~\cite{C2,C1}, we can define \defn{$q$-Stirling numbers of the second kind}, $S_q(n,k)$, by
\begin{equation}\label{eqn:qstirling_con}
	[z]_q^n = \sum_{k=0}^{n} q^{\binom{k}{2}} S_q(n,k) [z]_q [z-1]_q\dots [z-k+1]_q,
\end{equation}
where
\begin{equation}
	[z]_q \eqdef \frac{1-q^z}{1-q}
\end{equation}
is the $q$-number of $z$. 
We define \defn{$q$-falling factorials} by
\begin{equation}
z \da_n^q \eqdef z(z-[1]_q)\cdots (z-[n-1]_q),\text{ for $n>0$,} \quad z \da_0^q=1,
\end{equation}
and use the identity $$[m]_q - [n]_q = [m-n]_q q^n$$ to reformulate equation \eqref{eqn:qstirling_con} as
\begin{equation}\label{eqn:qstirling_con2}
	[z]_q^n = \sum_{k=0}^{n} S_q(n,k) [z]_q \da_k^q.
\end{equation}
One can derive from this relation that $S_q(n,k)$ satisfies
\begin{subequations} \label{eqn:qstirling_rec}
	\begin{align}
		S_q(n,k) &= 0, \qquad \text{for $k<0$ or $k>n$}, \\
		S_q(0,0) &= 1, \\
		\intertext{and, for $k \geq 0$,}
		S_q(n+1,k) &= S_q(n,k-1) + [k]_q S_q(n,k).
	\end{align}
\end{subequations}
In order to derive an explicit sum formula for these numbers, Carlitz generalizes the difference operator $\Delta$ by the recursive definition
\begin{subequations}
	\begin{align}
		\Delta_q p(z) &= p(z+1) - p(z),\\
		\Delta_q^{n+1} p(z) &= \Delta_q^n p(z+1) - q^{n} \Delta_q^n p(z)\\
		\intertext{for $n>0$, or explicitly by}
		\Delta_q^{n} &= \prod_{i=1}^{n} (E_z - q^{n-i}).
	\end{align}
\end{subequations}
Applying $\Delta_q^k$ to both sides of $\eqref{eqn:qstirling_con}$, Carlitz derives  (\cite[Equation (3.3)]{C1})
\begin{align}\label{eqn:qstirling_sum}
	S_q(n,k) = \frac{q^{-\binom{k}{2}}}{[k]_q!} \sum_{j=0}^{k} (-1)^j q^{\binom{j}{2}} \qBinom{k}{j} [k-j]_q^n,
\end{align}
where the \defn{$q$-factorials} are defined by 
\begin{equation}
	[n]_q! \eqdef \prod_{i=1}^{n} [i]_q \quad \text{for $n>0$ and} \quad [0]_q! = 1,
\end{equation} 
and the $q$-binomial coefficient is defined as
\begin{equation}
	\qBinom{n}{k} \eqdef \frac{[n]_q!}{[k]_q! [n-k]_q!}.
\end{equation}

In \cite{KS}, Keresk\'{e}nyin\'{e} Balogh and Schlosser study an elliptic analogue of Carlitz' $q$-Stirling numbers of the second kind. We introduce \defn{elliptic falling factorials} by the definition 
\[
z \da_n^{a,b;q,p} \eqdef z(z-[1]_{a,b;q,p})\cdots(z-[n-1]_{a,b;q,p}),
\]
and $z \da_0^{a,b;q,p}=1$, where $[z]_{a,b;q,p}$ is the elliptic number of $z$, which will be defined in the next section.
Then, the \defn{elliptic Stirling numbers of the second kind}, $S_{a,b;q,p}(n,k)$, can be defined by the relation
\begin{equation}\label{eqn:estirling_con}
	[z]_{a,b;q,p}^n = \sum_{k=0}^{n} S_{a,b;q,p}(n,k) [z]_{a,b;q,p} \da_k^{a,b;q,p}.
\end{equation}

In \cite{SYRook}, Schlosser and Yoo study elliptic Stirling numbers of the second kind in the context of rook theory. They ask for a generalization of the explicit sum formula \eqref{eqn:qstirling_sum} to $S_{a,b;q,p}(n,k)$ and work out special cases for $0 \leq k \leq 3$. Furthermore, the authors suggest using an appropriate extension of the $q$-difference operator $\Delta_q$ to achieve this goal.

In this article, we deduce an explicit formula for $S_{a,b;q,p}(n,k)$ analogous to \eqref{eqn:qstirling_sum} by using the fact that elliptic Stirling numbers of the second kind are a specialization of the complete homogeneous symmetric functions $h_k$. An explicit sum-formula for $h_k$ is well-known, and we obtain a formula for $S_{a,b;q,p}(n,k)$ as a special case. Furthermore, we will generalize Carlitz' $q$-difference operator appropriately to deduce a new operator proof of the formula for $h_k$. Special cases also contain Stirling numbers of type $B$ and $(q,r)$-Whitney numbers of the second kind.

We show that this operator proof can be generalized to prove explicit sum formulae for more general polynomials, including elliptic rook numbers and Lah numbers.

Another very interesting number sequence linked to Stirling numbers is the Eulerian number sequence. They were also generalized to the $q$-case by Carlitz~\cite{C3} but an elliptic extension of these numbers has not been defined so far. In Section~\ref{sec:eulerian} we study a generalization of Eulerian numbers similarly to the generalization of Stirling numbers of the second kind and derive explicit summation formulae. As special cases, we obtain elliptic Eulerian numbers and a $q$-analogue of $r$-Whitney Eulerian numbers. 

\section{Elliptic Stirling numbers}
In the past 15 years, several authors have studied elliptic function analogues of $q$-analogues of combinatorial objects and special combinatorial numbers \cite{BCK,Betea,BGR,KSY,Schl0,Schl1,Schl2,SYRook}. This has led to interesting extensions of weight functions assigned to several combinatorial models like lattice paths, lozenge tilings or rook placements. This study is closely connected to the theory of elliptic hypergeometric series, which naturally extends the theory of basic hypergeometric series. An introduction to elliptic hypergeometric series can be found in \cite[Chapter 11]{GR}, \cite{JosefDiss} and \cite{Ros}. In \cite{KS} and \cite{SYRook}, elliptic analogues of Stirling numbers, Lah numbers, and rook numbers were introduced. A main goal of this article is to find explicit summation formulae of these analogues. In this section, we will recall the most important notations and identities for the theory of elliptic functions. 

We begin by defining the \defn{modified Jacobi theta function} as
\[
\theta(x;p) \eqdef \prod_{j\geq 0} 
\left(
(1-p^jx)(1-\frac{p^{j+1}}{x})
\right),
\quad
\theta(x_1,\dots, x_\ell;p) = 
\prod_{k=1}^\ell \theta(x_k;p),
\]
where $x,x_1,\dots,x_\ell \neq 0$ and $|p|<1$.
The modified Jacobi theta function satisfies some fundamental
relations \cite[cf. p.~451, Example~5]{Web}, like the inversion formula
\begin{subequations}
	\begin{align}\label{eqn:ellinversion}
		\theta(x;p)=-x \theta(1/x;p),
	\end{align}
	the quasi-periodicity relation
	\begin{align}
		\theta(px;p)=-\frac{1}{x} \theta(x;p),
	\end{align}
	and the three-term identity
	\begin{align}\label{eqn:threeterm}
		\theta(xy,x/y,uz,u/z;p)=
		\theta(uy,u/y,xz,x/z;p)+
		\frac{x}{z}
		\theta(zy,z/y,ux,u/x;p).
	\end{align}
\end{subequations}
An \defn{elliptic function} is a function of a complex variable that is
meromorphic and doubly periodic, and it is well known that elliptic functions can
be written as quotients of modified Jacobi theta functions \cite{Ros,Web}.

An example, which we will use throughout this article, is the \defn{elliptic number} of a complex number $z$,
\begin{align}\label{def:ellnumber}
	\elliptic{z} \eqdef 
	\frac
	{\theta(q^z,aq^z,bq,aq/b;p)}
	{\theta(q,aq,bq^{z},aq^{z}/b;p)}.
\end{align}
Here, $a$ and $b$ are two additional parameters.\footnote{This definition follows the definition in \cite{logconcav_schlosser_2020} and corresponds to the definition in \cite{SYRook} subject to the substitution $b \mapsto bq^{-1}$.}
The elliptic number is a generalization of the $q$-number: Taking the limit $p\rightarrow 0$, then $a\rightarrow 0$ and then $b\rightarrow 0$, one recovers the $q$-analogue $[z]_q$.

The function $\elliptic{z}$ is elliptic in $a$, $b$, $q$ and $q^z$ each with the same multiplicative period $p$ (see also \cite[Remark~4]{SYRook}).  

We also define the elliptic weight
\begin{equation}
	W_{a,b;q,p}(k) \eqdef 
	\frac
	{\theta(aq^{2k+1},b,bq,a/b,aq/b;p)}
	{\theta(aq,bq^{k},bq^{k+1},aq^{k}/b,aq^{k+1}/b;p)}q^k,
\end{equation}
which is again an elliptic function and reduces to $q^k$ by taking the limit $p\rightarrow 0$, $a\rightarrow 0$ and $b\rightarrow 0$ (in this order). 

As an application of the three-term identity \eqref{eqn:threeterm} we obtain the following factorization for the elliptic number for complex values $y$ and $z$ \cite[Equation (3.8b)]{SYRook}:
\begin{align}\label{eqn:ellnumber_id}
	\elliptic{y+z} = \elliptic{y} + W_{a,b;q,p}(y) [z]_{aq^{2y},bq^{y};q,p}.
\end{align}
In this article we will also use the relations (\cite[Equation (3.3b)]{SYRook})
\begin{equation}\label{eqn:weightshift}
	W_{a,b;q,p}(k+j) = W_{a,b;q,p}(j)  W_{aq^{2j},bq^{j};q,p}(k)
\end{equation}
and (\cite[Corollary 4.3]{KSY})
\begin{equation}\label{eqn:minuselliptic}
	\elliptic{-k}=-W_{a,b;q,p}(-1)[k]_{1/a,b/a;q,p}. 
\end{equation}

Now we are ready to define elliptic Stirling numbers of the second kind. Recall the definition of \defn{elliptic falling factorials} 
\[
z \da_n^{a,b;q,p} \eqdef z(z-[1]_{a,b;q,p})\cdots(z-[n-1]_{a,b;q,p}) \quad \text{and} \quad z \da_0^{a,b;q,p}=1.
\]
and the definition of \defn{elliptic Stirling numbers of the second kind}, $S_{a,b;q,p}(n,k)$,
\begin{equation}
	[z]_{a,b;q,p}^n = \sum_{k=0}^{n} S_{a,b;q,p}(n,k) [z]_{a,b;q,p} \da_k^{a,b;q,p}.
\end{equation}
Note that each factor $[z]_{a,b;q,p}-[j]_{a,b;q,p}$ in the elliptic falling factorial of $[z]_{a,b;q,p}$ factorizes to
\begin{align*}
	[z]_{a,b;q,p}-[j]_{a,b;q,p} = W_{a,b;q,p}(j)[z-j]_{aq^{2j},bq^{j};q,p}
\end{align*}
by \eqref{eqn:ellnumber_id}.\footnote{\label{definition_note} This definition is slightly different to the definition in \cite{KS}. The original definition is obtained by multiplying $S_{a,b;q,p}(n,k)$ with $\prod_{j=0}^{k-1} W_{a,b;q,p}(j)$, but it turns out that it is convenient to stick to the definition \eqref{eqn:estirling_con}.}

From \eqref{eqn:estirling_con} one can deduce the following recurrence relation for $S_{a,b;q,p}(n,k)$:
\begin{subequations} \label{eqn:estirling_rec}
	\begin{align}
		S_{a,b;q,p}(n,k) &= 0, \qquad \text{for $k<0$ or $k>n$}, \\
		S_{a,b;q,p}(0,0) &= 1, \\
		\intertext{and, for $k \geq 0$,}
		S_{a,b;q,p}(n+1,k) &= S_{a,b;q,p}(n,k-1) + [k]_{a,b;q,p} S_{a,b;q,p}(n,k).
	\end{align}
\end{subequations}
In the next section, we will focus on a generalization of elliptic Stirling numbers, the complete homogeneous symmetric functions, before we return to the elliptic case. 
\section{Complete homogeneous symmetric functions}
Complete homogeneous symmetric functions, $h_k(a_0,a_1,a_2,\dots,a_n)$, are defined by
\[
h_k(a_0,a_1,a_2,\dots,a_n) = \sum_{0 \leq j_1 \leq j_2 \leq \dots \leq j_k \leq n} a_{j_1} a_{j_2} \dots a_{j_k}.
\]
We introduce a generalization of falling factorials and factorials for a series of variables $\va = a_0,a_1,a_2,\dots$ (following \cite{SS}) as
\begin{equation}
	\fall{z}{\va}{0} \eqdef 1, \quad \fall{z}{\va}{n} \eqdef \prod_{i=0}^{n-1} (z-a_i) \quad \text{and} \quad \fac{a}{n} \eqdef \fall{a_n}{\va}{n}.
\end{equation}
It is a well-known fact that $h_k$ satisfies the relation
\begin{align}\label{eqn:op_h_gf}
	z^n = \sum_{k=0}^{n} h_{n-k}(a_0,a_1,\dots,a_{k}) \fall{z}{\va}{k}.
\end{align}
This equation appears, for example, in \cite[Theorem 2.2]{SS} and \cite[p. 208]{DAL}.
By setting $a_i=[i]_q$ for all $i\geq 0$ and $z=[z]_q$ in this equation, we obtain the generating function \eqref{eqn:qstirling_con2} of Carlitz' $q$-Stirling numbers of the second kind. This leads to the specialization 
\[
S_q(n,k) = h_{n-k}([0]_q,[1]_q,\dots,[k]_q),
\]
which is also well known and appeared, for example, recently in \cite{SS}. 

By setting $a_i=[i]_{a,b;q,p}$ for all $i$ and $z=[z]_{a,b;q,p}$ in \eqref{eqn:op_h_gf}, we obtain the generating function for elliptic Stirling numbers of the second kind \eqref{eqn:estirling_con}, therefore
\begin{align}\label{eqn:estirling_h}
	S_{a,b;q,p}(n,k) = h_{n-k}([0]_{a,b;q,p},[1]_{a,b;q,p},\dots,[k]_{a,b;q,p}).
\end{align}
By multiplying both sides of \eqref{eqn:op_h_gf} with $z$ it follows that
\begin{align*}
	z^{n+1} &= \sum_{k=0}^{n} h_{n-k}(a_0,a_1,\dots,a_{k}) \fall{z}{\va}{k} (z-a_k+a_k)\\
	&= \sum_{k=0}^{n} h_{n-k}(a_0,a_1,\dots,a_{k}) \fall{z}{\va}{k+1} + \sum_{k=0}^{n} a_k h_{n-k}(a_0,a_1,\dots,a_{k}) \fall{z}{\va}{k} \\
	&= \sum_{k=0}^{n+1} \big( h_{n-k+1}(a_0,a_1,\dots,a_{k-1})+a_k h_{n-k}(a_0,a_1,\dots,a_{k}) \big) \fall{z}{\va}{k}
\end{align*}
and we can derive the recursion
\begin{subequations}
	\begin{align}
		h_{n-k}(a_0,a_1,\dots,a_{k}) &= 0 \quad \text{ for $k<0$ or $k>n$}, \\
		h_0(a_0) &= 1\\
		\intertext{and, for $k\geq 0$,}
		h_{n-k+1}(a_0,a_1,\dots,a_{k}) &= h_{n-k+1}(a_0,a_1,\dots,a_{k-1})+a_k h_{n-k}(a_0,a_1,\dots,a_{k})
	\end{align}
\end{subequations}
analogous to \eqref{eqn:stirling_rec}, \eqref{eqn:qstirling_rec} and \eqref{eqn:estirling_rec}. 

The complete homogeneous symmetric functions can be computed by the following explicit formula (\cf \cite{GM}, where this formula was obtained as a special case of a Schur function identity, or \cite[Exercise 7.4]{Stanley2}).

\begin{theorem}\label{thm:h_expl}
	Let $n,k$ be nonnegative integers, $a_0,a_1,\dots,a_k$ be variables and $h_n(a_0,\dots,a_k)$ the $n$-th complete homogeneous symmetric function in $k+1$ variables. We then have
	\begin{equation}\label{eqn:h_expl}
		h_{n}(a_0,\dots,a_k) = \sum_{j=0}^{k}\Big( \prod_{\substack{i=0\\ i\neq j}}^{k} (a_j-a_i)^{-1} \Big) a_{j}^{n+k}. 
	\end{equation}
\end{theorem}

By setting $a_i = i$, respectively $a_i = [i]_q$, for all $i$ and doing some small manipulations, we recover the sum-formulae for $S(n,k)$ in \eqref{eqn:stirling_exp} and $S_q(n,k)$ in \eqref{eqn:qstirling_sum}, respectively. 

To complete the task of finding an explicit formula for $S_{a,b;q,p}(n,k)$, described by Schlosser and Yoo \cite{SYRook}, we set $a_i=\elliptic{i}$, change the indices according to \eqref{eqn:estirling_h} and reverse the order of summation, to obtain the following corollary.

\begin{corollary}
	Let $n,k$ be nonnegative integers, $\elliptic{n}$ be the elliptic number of $n$ and $S_{a,b;q,p}(n,k)$ be the elliptic Stirling numbers of the second kind defined by \eqref{eqn:estirling_con}. We then have
	\begin{equation}
		S_{a,b;q,p}(n,k) = \sum_{j=0}^{k} \Big( \prod_{\substack{i=0\\ i\neq k-j}}^{k} (W_{a,b;q,p}(i)[k-j-i]_{aq^{2i},bq^i;q,p})^{-1} \Big) \elliptic{k-j}^n. 
	\end{equation}
\end{corollary}

This formula reduces to \eqref{eqn:qstirling_sum} by taking the limits $p\rightarrow 0$, then $a\rightarrow 0$ and then $b\rightarrow 0$.

\begin{remark}
	In contrast to \eqref{eqn:stirling_exp} and \eqref{eqn:qstirling_sum} it seems that there appears no binomial coefficient in the sum-formula for $h_k$ \eqref{eqn:h_expl}, but we can define an $\va$-binomial coefficient for $0 \leq k \leq n$ as
	\[
	\eBinom{n}{k}{\va} \eqdef (-1)^{n-k} \frac{a_n!}{\prod_{\substack{i=0\\ i\neq k}}^{n} (a_k-a_i)} = (-1)^{n-k+1} \prod_{\substack{i=0\\ i\neq k}}^{n-1} \frac{a_n-a_i}{a_k-a_i}
	\]
	From \eqref{eqn:h_expl} we obtain
	\begin{equation}\label{eqn:h_expl2}
		h_{n}(a_0,\dots,a_k) = \frac{1}{a_k!}\sum_{j=0}^{k} (-1)^{k-j} \eBinom{k}{j}{\va}  a_{j}^{n+k}.
	\end{equation}
	This binomial coefficient fulfils the recurrence relation 
	\begin{equation} \label{eqn:xbinom_rec}
		\eBinom{n+1}{k}{\va} = \prod_{j=0}^{n-1} \frac{a_{n+1}-a_{j+1}}{a_n-a_j} \eBinom{n}{k}{\va} + E_{\va} \eBinom{n}{k-1}{\va},
	\end{equation}
	for $0<k< n+1$,	where $E_{\va}$ is an operator that maps all $a_i$ to $a_{i+1}$. 
	However, we chose not to use this notation in this work. 
\end{remark}

\section{A generalized difference operator}

We consider functions $f(z;\vx)$ in the variables $z$ and $\vx=x_0, x_1,x_2\dots$ and recall the generalized notation for the falling factorial and the factorial according to $z$ and $\vx$ as follows: 
\begin{equation}
	\fall{z}{\vx}{0} = 1, \quad \fall{z}{\vx}{n}=\prod_{i=0}^{n-1} (z-x_i) \quad \text{for $n>0$ and} \quad \fac{x}{n}=\fall{x_n}{\vx}{n}.
\end{equation}
We define the \defn{shift operator} $E_{z,\vx}$ as $E_{z,\vx} f(z;x_0,x_1,x_2\dots)=f(z+1;x_1,x_2,x_3\dots)$.  For example, $$E_{z,\vx} \frac{z-x_1+a_2}{x_k-x_2} = \frac{z+1-x_2+a_2}{x_{k+1}-x_3}.$$ The \defn{difference operator} $\Delta_{\vx}^n$ is recursively defined by
\begin{subequations}
	\begin{align}
		\Delta_{\vx} f(z,\vx)&=E_{z,\vx} f(z,\vx) - f(z,\vx),\\
		\Delta_{\vx}^{n+1} f(z,\vx) &= E_{z,\vx} \Delta_{\vx}^n f(z,\vx) - \Big( \prod_{j=0}^{n-1} \frac{x_{n+1}-x_{j+1}}{x_{n}-x_{j}} \Big) \Delta_{\vx}^n f(z,\vx)\\
		\intertext{for $n>0$, or explicitly by}
		\Delta_{\vx}^{n} &= \prod_{i=0}^{n-1} \Big( E_{z,\vx}-\prod_{j=0}^{n-i-2} \frac{x_{n-i}-x_{j+1}}{x_{n-i-1}-x_{j}} \Big).
	\end{align}
\end{subequations}
It follows by induction on $n$ (and \eqref{eqn:xbinom_rec}) that
\begin{align}
	\Delta_{\vx}^{n} = \sum_{k=0}^{n} \fac{x}{n} \Big( \prod_{\substack{i=0\\ i\neq k}}^{n} (x_k-x_i)^{-1} \Big) E_{z,\vx}^k.
\end{align}


We are now ready to prove the explicit formula for $h_n(a_0,a_1,\dots,a_k)$. 
\begin{proof}[Proof of Theorem~\ref{thm:h_expl}]
	Let $C_k$ be the connection coefficients in
	\begin{equation}\label{eqn:beta}
		f(z)=\sum_{k=0}^n C_{k} \fall{a_z}{\va}{k},
	\end{equation}
	where $f(z)$ is a polynomial in $a_z$ of degree at most $n$ and independent of the variables $\vx$. Then the $C_k$ are independent of $z$ and $\vx$.
	We apply $\Delta_{\vx}^j$ to both sides of \eqref{eqn:beta} and set $z=0$ and $x_i=a_i$ for $i\geq 0$ afterwards to obtain
	\[
	\Delta_{\vx}^j f(z) \big|_{z=0,x_i=a_i} = \sum_{k=0}^n C_{k} \Delta_{\vx}^j \fall{a_z}{\va}{k} \big|_{z=0,x_i=a_i}
	\]
	since $C_{k}$ is independent of $z$ and $x_i$. The above notation means that we set $z=0$ and $x_i=a_i$ for $i\geq 0$ after applying the operators to both sides. Evaluating the operator on the right hand side gives
	\[
	\Delta_{\vx}^j \fall{a_z}{\va}{k} = \sum_{s=0}^{j} \fac{x}{j} \prod_{\substack{i=0\\ i\neq s}}^{j} (x_s-x_i)^{-1} \prod_{i=0}^{k-1} (a_{z+s}-a_{i}).
	\]
	In this expression we can set $z=0$ and $x_i=a_i$ for $i\geq 0$ to obtain
	\begin{align*}
		\Delta_{\vx}^j \fall{a_z}{\va}{k} \big|_{z=0,x_i=a_i} = \sum_{s=0}^{j} \fac{a}{j} \frac{\prod_{i=0}^{k-1} (a_{s}-a_{i})}{\prod_{\substack{i=0\\ i\neq s}}^{j} (a_s-a_i)}.
	\end{align*}
	We want to show that this sum simplifies to $\delta_{j,k} \fac{a}{j}$, where $\delta$ is the Kronecker delta.
	
	Suppose $0 \leq j < k$. Then we have a factor $(a_s-a_s)=0$  in the numerator of each summand, and the sum vanishes.
	
	For $k=j$ the sum reduces to $a_k!$ by the same argument.
	
	Suppose $k<j\leq n$. Then, 
	\begin{align*}
		\sum_{s=0}^{j} \fac{a}{j} \frac{\prod_{i=0}^{k-1} (a_{s}-a_{i})}{\prod_{\substack{i=0\\ i\neq s}}^{j} (a_s-a_i)}=\sum_{s=k}^{j} \fac{a}{j} \frac{\prod_{i=0}^{k-1} (a_{s}-a_{i})}{\prod_{\substack{i=0\\ i\neq s}}^{j} (a_s-a_i)}= \sum_{s=k}^{j} \fac{a}{j} \prod_{\substack{i=k\\ i\neq s}}^{j} (a_s-a_i)^{-1}.
	\end{align*}
	In order to prove
	\[
	\sum_{s=k}^{j} \fac{a}{j} \prod_{\substack{i=k\\ i\neq s}}^{j} (a_s-a_i)^{-1} = 0
	\]
	for $k<j$, we assume $\fall{a_j}{\va}{k}\neq 0$ (otherwise the equation would be true as well) and divide both sides by $\fall{a_j}{\va}{k}$. By a simple computation we arrive at the equivalent equation
	\[
	\sum_{s=k}^{j-1} \prod_{\substack{i=k\\i\neq s}}^{j-1}  \frac{(a_{j}-a_i)}{(a_s-a_i)} = 1.
	\]
	The sum on the left side can be seen as a polynomial in $a_{j}$ of degree $j-k-1$ which is equal to $1$ in the $j-k$ values $a_k, a_{k+1}, \dots, a_{j-1}$ inserted for $a_{j}$. In fact, it is the Lagrange interpolation polynomial for the function $f(a_j)=1$. Therefore, the equality holds and we have proven the case $k<j$.
	
	In summary, we have shown that 
	\[
	\frac{\Delta_{\vx}^j f(z)}{\fac{a}{j} } \big|_{z=0,x_i=a_i} = C_{j}.
	\]
	To prove the theorem, we take $f(z)=a_z^n$ to obtain 
	\begin{equation}\label{eqn:Ck}
	C_{k} = \frac{\Delta_{\vx}^k a_z^n}{\fac{a}{k}}\big|_{z=0,x_i=a_i} = \sum_{j=0}^{k}\Big( \prod_{\substack{i=0\\ i\neq j}}^{k} (a_j-a_i)^{-1} \Big) a_{j}^n.
	\end{equation}
	By comparing \eqref{eqn:beta} with \eqref{eqn:op_h_gf} (where $z$ is replaced by $a_z$) we finally arrive at
	\[
	h_{n}(a_0,\dots,a_k) = \sum_{j=0}^{k}\Big( \prod_{\substack{i=0\\ i\neq j}}^{k} (a_j-a_i)^{-1} \Big) a_{j}^{n+k}. \qedhere
	\]
\end{proof}
\begin{remark}[Specializations]
	Besides the elliptic and $q$-Stirling numbers, there are several interesting specializations of $h_n(a_0,a_1,\dots,a_k)$. For example, we obtain $(q,r)$-Whitney numbers of the second kind, $W_{m,r}[n,k]_q$ (see for example \cite{COL} or \cite{Mezo} for $q=1$), by
	$$W_{m,r}[n,k]_q q^{-kr-m\binom{k}{2}}= W_{m,r}^{\ast}[n,k]_q = h_{n-k}([r]_q,[m+r]_q,[2m+r]_q,\dots,[km+r]_q).$$
	Then, Theorem~\ref{thm:h_expl} reduces to \cite[Equation (12)]{COL}. 
	
	Remmel and Wachs defined a $(p,q)$-analogue of these numbers, the shifted Stirling numbers of the second kind. We obtain the $(p,q)$-shifted Stirling numbers of the second kind, $S_{n,k}^{r,m}(p,q)$ \cite[Equation (28)]{RW}, by $$S_{n,k}^{r,m}(p,q) = h_{n-k}([r]_{p,q},[m+r]_{p,q},[2m+r]_{p,q},\dots,[km+r]_{p,q}),$$ which are a $(p,q)$-analogue of $r$-Whitney numbers. Then, one can check by a careful computation that Theorem~\ref{thm:h_expl} reduces to \cite[Theorem 13]{RW}.
	
	An elliptic analogue of $S_{n,k}^{r,m}(p,q)$ was defined by Schlosser and Yoo~\cite{SYRook} and can be obtained by specializing $$S_{a,b;q,p}^{r,m}(n,k) = h_{n-k}([r]_{a,b;q,p},[m+r]_{a,b;q,p},[2m+r]_{a,b;q,p},\dots,[km+r]_{a,b;q,p}),$$ up to a power of elliptic weights (comparable to the remark in footnote \ref{definition_note}).
\end{remark}

\section{More general cases}
The proof of Theorem~\ref{thm:h_expl} can be adapted to any polynomial $f$ in $a_z$ of degree at most $n$ which is independent of the variables $\vx$. Let $C$ and $c_1,c_2,\dots,c_{n}$ be complex numbers. If we take $f(z)=c_0(a_z-c_1)(a_z-c_2)\cdots(a_z-c_n)$, then \eqref{eqn:beta} becomes
\begin{equation}\label{eqn:beta2}
	c_0 \prod_{i=1}^{n} (a_z-c_i)=\sum_{k=0}^n C_{n,k} \fall{a_z}{\va}{k},
\end{equation}
and, following the proof of Theorem~\ref{thm:h_expl}, the coefficients have the explicit form
\begin{equation}\label{eqn:general_alpha}
	C_{n,k} = \frac{\Delta_{\vx}^k f(z)}{\fac{a}{k}}\big|_{z=0,x_i=a_i} = c_0 \sum_{j=0}^{k}\Big( \prod_{\substack{i=0\\ i\neq j}}^{k} (a_j-a_i)^{-1} \Big) \prod_{i=1}^{n} (a_j-c_i).
\end{equation}
If we see $a_0, a_1, \dots, a_n$ as complex numbers, then $C_{n,k}$ is the change of basis coefficient between the two series of polynomials over $\C$
\[
\Big( c_0 \prod_{i=1}^{n} (a_z-c_i) \Big)_{n \geq 0} \qquad \text{and} \qquad \Big( \prod_{i=0}^{n-1} (a_z-a_i) \Big)_{n \geq 0}.
\]
From \eqref{eqn:beta2} we can derive the recurrence relations
\begin{subequations} \label{eqn:alpha_rec}
	\begin{align}
		C_{n,k} &= 0, \qquad \text{for $k<0$ or $k>n$ }, \\
		C_{0,0} &= c_0, \\
		\intertext{and, for $k \geq 0$,}
		C_{n+1,k} &= C_{n,k-1} + (a_k - c_{n+1}) C_{n,k}.
	\end{align}
\end{subequations}

\begin{remark}
	Note that the results discussed in this section are well-studied in other contexts. 
	For $c_0=1$, the coefficient $C_{n,k}$ is equal to the complementary symmetric function $\comp_k(\{c_1,c_2,\dots,c_n | a_0,a_1,\dots,a_{k}\})$ defined over hybrid sets in \cite{DAL}. This can be seen by comparing \eqref{eqn:beta2} with \cite[Corollary 2]{DAL}. 
	
	We can also see the results from the perspective of Newton polynomials and divided differences. The falling factorials $\fall{z}{\vx}{n}=\prod_{i=0}^{n-1} (z-x_i)$ are also known as Newton polynomials, and \eqref{eqn:Ck} is the explicit formula for the divided differences of $z^n$ (see for example \cite[p. 7]{finiteDifferences}). By linearity, this formula extends to arbitrary polynomials in $z$, yielding \eqref{eqn:general_alpha}.
\end{remark}

\subsection{Rook numbers}
In \cite{SYRook}, Schlosser and Yoo define the $k$-th elliptic rook number, $r_k(a,b;q,p;B)$, for a Ferrers board $$B=B(b_1,b_2,\dots,b_n)=\{(i,j): 1 \leq i \leq n, 1 \leq j \leq b_i\}$$
as the sum of specific elliptic weights for rook placements of $B$. We neglect the details here and only focus on the following algebraic expression. The authors obtain the following generating function for the elliptic rook numbers \cite[Theorem 12]{SYRook}:
\begin{multline} \label{eqn:erook_con}
	\prod_{i=1}^{n} [z-i+1+b_i]_{aq^{2(i-1-b_i)},bq^{i-1-b_i};q,p}\\
	=\sum_{k=0}^n r_{n-k}(a,b;q,p;B) \prod_{j=1}^{k} [z-j+1]_{aq^{2(j-1)},bq^{j-1};q,p}.
\end{multline}
We can reformulate this equation by using \eqref{eqn:ellnumber_id} as
\begin{multline}
	\prod_{i=1}^{n} W_{a,b;q,p}(i-1-b_i)^{-1}(\elliptic{z}-\elliptic{i-1-b_i})\\
	=\sum_{k=0}^n r_{n-k}(a,b;q,p;B) \Big( \prod_{j=1}^{k}W_{a,b;q,p}(j-1)^{-1} \Big) \elliptic{z} \da_{k}^{a,b;q,p}.
\end{multline}
Comparing this expression with \eqref{eqn:beta2} yields
\begin{equation}\label{eqn:erook_expl}
	r_{n-k}(a,b;q,p;B) = 
		\sum_{j=0}^{k} \frac{
			W_{a,b;q,p}(j) }{	W_{a,b;q,p}(k)
		} \frac{\prod_{i=1}^{n} [j-i+b_i+1]_{aq^{2(i-1-b_i)},bq^{i-1-b_i};q,p}}{\prod_{\substack{i=0\\ i\neq j}}^{k} [j-i]_{aq^{2i};bq^i;q,p}} 
\end{equation}
as a special case of \eqref{eqn:general_alpha} with $c_0=\prod_{i=1}^{n} W_{a,b;q,p}(i-1-b_i)$, $a_i = \elliptic{i}$, $c_i=\elliptic{i-1-b_i}$ and $$C_{n,k} = r_{n-k}(a,b;q,p;B)  \prod_{j=1}^{k}W_{a,b;q,p}(j-1)^{-1}.$$

These elliptic rook numbers generalize elliptic Stirling numbers of the second kind, which can be recovered by considering the staircase Ferrers board $B(1,2,3,\dots,n-1)$ (with the original definition of elliptic Stirling numbers). 

In the case of $p\to 0$, $b \to 0$ and $B=[n+r-1] \times [n-r]$ the identity \eqref{eqn:erook_expl} can be evaluated in closed form by virtue of the $q$-Pfaff--Saalsch\"utz summation, thus yielding a closed form of the respective $r$-restricted $a;q$-Lah numbers (compare to \cite[Subsection 3.1.1]{SYRook2}).

\subsection{Lah numbers}
Another interesting board considered in \cite{SYRook} is $B(b_1,b_2,\dots,b_n)=B(n-1,n-1,\dots,n-1)$. In this case, the elliptic rook numbers are equal to elliptic Lah numbers, defined in \cite{KS}. We will again use a slightly different definition (up to a power of elliptic weights). We introduce \defn{elliptic rising factorials} by the definition 
\[
z \ua_n^{a,b;q,p} \eqdef z(z-[-1]_{a,b;q,p})\cdots(z-[-n+1]_{a,b;q,p}),
\]
for $n>0$ and $z \ua_0^{a,b;q,p}=1$, and define \defn{elliptic Lah numbers}, $L_{a,b;q,p}(n,k)$, by
\begin{equation}\label{eqn:eLah_con}
	[z]_{a,b;q,p} \ua_n^{a,b;q,p} = \sum_{k=0}^{n} L_{a,b;q,p}(n,k) [z]_{a,b;q,p} \da_k^{a,b;q,p}.
\end{equation}
They fulfil the recurrence relation
\begin{subequations} \label{eqn:elah_rec}
	\begin{align}
		L_{a,b;q,p}(n,k) &= 0, \qquad \text{for $k<0$ or $k>n$ }, \\
		L_{a,b;q,p}(0,0) &= 1, \\
		\intertext{and, for $k \geq 0$,}
		L_{a,b;q,p}(n+1,k) &= L_{a,b;q,p}(n,k-1) + W_{a,b;q,p}(-n) [n+k]_{aq^{-2n},bq^{-n};q,p} L_{a,b;q,p}(n,k).
	\end{align}
\end{subequations}
We derive the explicit expression
\begin{equation}
	L_{a,b;q,p}(n,k) =
	\sum_{j=0}^{k} \frac{\prod_{i=1}^{n} (\elliptic{j}-\elliptic{-n+i})}{\prod_{\substack{i=0\\ i\neq j}}^{k} (\elliptic{j}-\elliptic{i})} 
\end{equation}
by comparing \eqref{eqn:eLah_con} with \eqref{eqn:erook_con} or directly with \eqref{eqn:beta2}.

\section{Eulerian numbers}\label{sec:eulerian}
The \defn{Eulerian numbers}, $A(n,k)$, play an important role in combinatorics. We follow the notation of Carlitz \cite{C3} and define Eulerian numbers by the recurrence relation
\begin{subequations} \label{eqn:eulerian_rec}
	\begin{align}
		A(n,k) &= 0, \qquad \text{for $k<0$ or $k>n$}, \\
		A(0,0) &= 1, \\
		\intertext{and, for $k \geq 0$,}
		A(n+1,k) &= (n-k+2) A(n,k-1) + k A(n,k).
	\end{align}
\end{subequations}
The Eulerian numbers can also be defined as the connection coefficients in the identity
\begin{align}\label{eqn:eulerian_con}
	z^n = \sum_{k=0}^{n} A(n,k) \binom{z+k-1}{n},
\end{align}
which is known as Worpitzky's identity. Besides several other properties, Eulerian numbers can be expressed by the explicit sum formula
\begin{align}
	\label{eqn:eulerian_exp}
	A(n,k) = \sum_{j = 0}^{k} (-1)^j \binom{n+1}{j} (k-j)^n.
\end{align}
We recommend \cite{EN} for an introduction to Eulerian numbers. 

Carlitz introduced a $q$-analogue of Eulerian numbers \cite{C3}, the \defn{$q$-Eulerian numbers}, $A_q(n,k)$, and defined them as the connection coefficients in 
\begin{align}\label{eqn:q_eulerian_con}
	[z]_q^n = \sum_{k=0}^{n} A_q(n,k) \qBinom{z+k-1}{n}.
\end{align}
He derives from this relation that $A_q(n,k)$ satisfies
\begin{subequations} \label{eqn:qeulerian_rec}
	\begin{align}
		A_q(n,k) &= 0, \qquad \text{for $k<0$ or $k>n$ }, \\
		A_q(0,0) &= 1, \\
		\intertext{and, for $k \geq 0$,}
		A_q(n+1,k) &= [n-k+2]_q A_q(n,k-1) + q^{n-k+1} [k]_q A_q(n,k)
	\end{align}
\end{subequations}
and the explicit formula (see \cite[Equation (5.5), corrected]{C3})
\begin{align}\label{eqn:qeulerian_sum}
	A_q(n,k) = q^{\binom{n-k+1}{2}-\binom{k}{2}} \sum_{j=0}^{k} (-1)^{j} q^{\binom{j}{2}} \qBinom{n+1}{j} [k-j]_q^n.
\end{align}

\subsection{Generalized Eulerian numbers}
It turns out that it is possible to generalize Carlitz' $q$-Eulerian numbers similarly to the generalizations in the previous sections to elliptic Eulerian numbers and a Laurent polynomial analogue. We begin by extending $\va$ to variables with negative index by $\va \eqdef \dots,a_{-2},a_{-1},a_0,a_1,a_2,\dots$ and define generalized Eulerian numbers, $A_\va(n,k)$, as the connection coefficients in
\begin{align}\label{eqn:a_eulerian_con}
	z^n = \sum_{k=0}^{n} A_\va(n,k) \prod_{i=1}^{n} \frac{z-a_{i-k}}{a_{n-k+1}-a_{i-k}}.
\end{align}
This definition was chosen due to the fact that these coefficients can be characterized by a nice recursion, and we will derive a simple explicit sum formula for $A_\va(n,k)$. 

Note that $A_\va(n,k)=0$ for $k<0$ or $k>n$ and $A_\va(0,0)=1$.
We can derive a recursion for $A_\va(n,k)$ by using the identity
\begin{equation}\label{eqn:rec_hilfe}
	z = \frac{a_{n-k+1}(z-a_{-k})-a_{-k}(z-a_{n-k+1})}{a_{n-k+1}-a_{-k}}
\end{equation}
and we introduce the notation
\begin{equation}
	P(n,k) \eqdef \prod_{i=1}^{n+1} \frac{a_{n-k+2}-a_{i-k}}{a_{n-k+1}-a_{i-1-k}}.
\end{equation}
By multiplying $z$ to both sides of \eqref{eqn:a_eulerian_con} using \eqref{eqn:rec_hilfe} we obtain
\begin{align*}
	z^{n+1} &= \sum_{k=0}^{n} A_\va(n,k) \Big(\prod_{i=1}^{n} \frac{z-a_{i-k}}{a_{n-k+1}-a_{i-k}}\Big) \frac{a_{n-k+1}(z-a_{-k})-a_{-k}(z-a_{n-k+1})}{a_{n-k+1}-a_{-k}} \\	
	&=\sum_{k=0}^{n} A_\va(n,k) a_{n-k+1}  \prod_{i=0}^{n} \frac{z-a_{i-k}}{a_{n-k+1}-a_{i-k}} \\	
	&\qquad -\sum_{k=0}^{n} A_\va(n,k) a_{-k}  \prod_{i=0}^{n} \frac{z-a_{i-k+1}}{a_{n-k+1}-a_{i-k}} \\	
	&=\sum_{k=1}^{n+1} A_\va(n,k-1) a_{n-k+2}  \prod_{i=1}^{n+1} \frac{z-a_{i-k}}{a_{n-k+2}-a_{i-k}} \\	
	&\qquad -\sum_{k=0}^{n} A_\va(n,k) a_{-k}  \Big( \prod_{i=1}^{n+1} \frac{a_{n-k+2}-a_{i-k}}{a_{n-k+1}-a_{i-1-k}} \Big) \prod_{i=1}^{n+1} \frac{z-a_{i-k}}{a_{n-k+2}-a_{i-k}} \\	
	&=\sum_{k=0}^{n+1} \Big(a_{n-k+2} A_\va(n,k-1) - a_{-k} P(n,k) A_\va(n,k)\Big)  \prod_{i=1}^{n+1} \frac{z-a_{i-k}}{a_{n-k+2}-a_{i-k}}.
\end{align*}
By comparing the last sum with the definition of $A_\va(n+1,k)$ in \eqref{eqn:a_eulerian_con} we obtain the recursive description for the generalized Eulerian numbers
\begin{subequations} \label{eqn:a_eulerian_rec}
	\begin{align}
		A_\va(n,k) &= 0, \qquad \text{for $k<0$ or $k>n$}, \\
		A_\va(0,0) &= 1, \\
		\intertext{and, for $k \geq 0$,}
		A_\va(n+1,k) &=a_{n-k+2} A_\va(n,k-1) - a_{-k} P(n,k) A_\va(n,k).
	\end{align}
\end{subequations}
Note that in comparison to the generalized Stirling, Lah and rook numbers, $A_\va(n,k)$ is not a polynomial in the variables $\va$ due to the factor $P(n,k)$ in the recursion. 
However, $P(n,k)$ factorizes to $1$ or a power of $q$ in the following specializations.
\begin{remark}[Specializations]
	By setting $a_i=i$ for all $i \in \Z$, we obtain $P(n,k)=1$ and $A_\va(n,k)$ reduces to the ordinary Eulerian numbers $A(n,k)$.
	
	By setting $a_i=[i]_q$ for all $i \in \Z$, we obtain $P(n,k)=q^{n+1}$, $a_{-k}=[-k]_q = - q^{-k}[k]_q$, thus we can see by comparing the recursion formulae \eqref{eqn:a_eulerian_rec} and \eqref{eqn:qeulerian_rec} (or by comparing the definitions) that $A_\va(n,k)$ reduces to Carlitz' $q$-Eulerian numbers $A_q(n,k)$.
	
	We study more specializations in Section~\ref{sec:eulerian_spec}.
\end{remark}
\subsection{Explicit formula for $A_\va(n,k)$}
In this subsection, we derive an explicit formula for the generalized Eulerian numbers.
\begin{theorem}\label{thm:aeulerian_sum}
	For $n \geq 0$ and $0 \leq k \leq n$ there holds
	\begin{equation}\label{eqn:A_explicit}
		A_\va(n,k) = \sum_{j=0}^{k} \Big( \prod_{\substack{i=0\\ i\neq k-j}}^{n} \frac{a_{n-k+1}-a_{i-k}}{a_{-j}-a_{i-k}} \Big) a_{-j}^n.
	\end{equation}
\end{theorem}

\begin{proof}
	We begin by letting $f(z)$ be a polynomial in $a_z$ of degree at most $n$ and let $C_{n,k}$ be the connection coefficients in
	\begin{equation}\label{eqn:f_eul_con}
		f(z) = \sum_{k=0}^{n} C_{n,k} \prod_{i=1}^{n} \frac{a_z-a_{i-k}}{a_{n-k+1}-a_{i-k}}
	\end{equation}
	and we claim that
	\begin{equation}\label{eqn:f_eul_expl}
		C_{n,k} = \sum_{j=0}^{k} \Big( \prod_{\substack{i=0\\ i\neq k-j}}^{n} \frac{a_{n-k+1}-a_{i-k}}{a_{-j}-a_{i-k}} \Big) f(-j).
	\end{equation}
	We prove this claim by inserting \eqref{eqn:f_eul_con} into the right side of \eqref{eqn:f_eul_expl} to derive
	\begin{align*}
		 &\sum_{j=0}^{k} \Big( \prod_{\substack{i=0\\ i\neq k-j}}^{n} \frac{a_{n-k+1}-a_{i-k}}{a_{-j}-a_{i-k}} \Big) f(-j) \\
		 &=\sum_{j=0}^{k} \Big( \prod_{\substack{i=0\\ i\neq k-j}}^{n} \frac{a_{n-k+1}-a_{i-k}}{a_{-j}-a_{i-k}} \Big) \sum_{l=0}^{n} C_{n,l} \prod_{i=1}^{n} \frac{a_{-j}-a_{i-l}}{a_{n-l+1}-a_{i-l}}\\
		  &=\sum_{l=0}^{n} C_{n,l} \sum_{j=0}^{k} \Big( \prod_{\substack{i=0\\ i\neq k-j}}^{n} \frac{a_{n-k+1}-a_{i-k}}{a_{-j}-a_{i-k}} \Big) \prod_{i=1}^{n} \frac{a_{-j}-a_{i-l}}{a_{n-l+1}-a_{i-l}} \\
		  &= \sum_{l=0}^{k} C_{n,l} \sum_{j=l}^{k} \Big( \prod_{\substack{i=0\\ i\neq k-j}}^{n} \frac{a_{n-k+1}-a_{i-k}}{a_{-j}-a_{i-k}} \Big) \prod_{i=1}^{n} \frac{a_{-j}-a_{i-l}}{a_{n-l+1}-a_{i-l}}
	\end{align*}	
	We prove that the last expression is equal to $C_{n,k}$ by showing that
	\begin{equation}\label{eqn:euler_proof_delta}
		\sum_{j=l}^{k} \Big( \prod_{\substack{i=0\\ i\neq k-j}}^{n} \frac{a_{n-k+1}-a_{i-k}}{a_{-j}-a_{i-k}} \Big) \prod_{i=1}^{n} \frac{a_{-j}-a_{i-l}}{a_{n-l+1}-a_{i-l}} = \delta_{k,l}
	\end{equation}
	for $0 \leq l \leq k \leq n$. First, we factor out 
	\begin{equation}\label{eqn:euler_proof_lagrange}
		\frac{ \prod_{i=k-l+1}^{n} (a_{n-k+1}-a_{i-k})}{ \prod_{i=1}^{n} (a_{n-l+1}-a_{i-l})}
		\sum_{j=l}^{k} \Big( \prod_{\substack{i=0\\ i\neq k-j}}^{k-l} \frac{a_{n-k+1}-a_{i-k}}{a_{-j}-a_{i-k}} \Big) \prod_{i=1}^{k-l} (a_{-j}-a_{i+n-k})
	\end{equation}
	to obtain that the sum in this expression is the Lagrange interpolation polynomial of the polynomial
	\[
		g(z) = \prod_{i=1}^{k-l} (z-a_{i+n-k}),
	\]
	for the $k-l+1$ values $a_{-l}, a_{-l-1}, \dots, a_{-k}$ inserted for $z$ and $z$ is replaced by $a_{n-k+1}$. Since $g(z)$ is a polynomial in the variable $z$ of degree $k-l$, the Lagrange polynomial and $g(z)$ are equal, such that \eqref{eqn:euler_proof_lagrange} is equal to
	\begin{equation}
	\frac{ \prod_{i=k-l+1}^{n} (a_{n-k+1}-a_{i-k})}{ \prod_{i=1}^{n} (a_{n-l+1}-a_{i-l})} \prod_{i=1}^{k-l} (a_{n-k+1}-a_{i+n-k}) = \delta_{k,l}.
	\end{equation}
	We therefore proved \eqref{eqn:euler_proof_delta} and \eqref{eqn:f_eul_expl}. By comparing \eqref{eqn:f_eul_con} with \eqref{eqn:a_eulerian_con} and setting $f(z) = a_z^n$, we complete the proof of the theorem.
\end{proof}
\subsection{Specializations}\label{sec:eulerian_spec}
In the literature, several generalizations of Eulerian numbers and $q$-Eulerian numbers were studied. It appears that some generalizations are special cases of $A_\va (n,k)$. 

For example, the \defn{$r$-Whitney Eulerian numbers} \cite{MRSV, RVV} satisfy (see \cite[Equation (6)]{RVV})
\begin{subequations} \label{eqn:r_W_eulerian_rec}
	\begin{align}
		A_{m,r}(n,k) &= 0, \qquad \text{for $k<0$ or $k>n$}, \\
		A_{m,r}(0,0) &= 1, \\
		\intertext{and, for $k \geq 0$,}
		A_{m,r}(n+1,k) &=(m(n-k+2)-r) A_{m,r}(n,k-1) + (mk+r) A_{m,r}(n,k).
	\end{align}
\end{subequations}
The $r$-Whitney Eulerian numbers reduce to classical Eulerian numbers for $(m,r)=(1,0)$ and to the $q$-Eulerian numbers due to Brenti \cite{Brenti} for $(m,r)=(q+1,1)$. By comparing \eqref{eqn:r_W_eulerian_rec} with \eqref{eqn:a_eulerian_rec} one can see that setting $a_i = m i - r$ in $A_\va(n,k)$ yields $A_{m,r}(n,k)$, since the polynomials $P(n,k)$ reduce to $1$ in this case. Furthermore, for $z=mx-r$ the Worpitzky identity \eqref{eqn:a_eulerian_con} reduces to 
\[
	(mx-r)^n = \sum_{k=0}^n A_{m,r}(n,k) \binom{x+k-1}{n}
\]
which is in fact different to the Worpitzky identity given in \cite[Theorem 5]{RVV}. The explicit formula \eqref{eqn:A_explicit} yields \cite[Theorem 6]{RVV}.

A $q$-analogue of the $r$-Whitney Eulerian numbers, $A_q^{m,r}(n,k)$, can be obtained by setting $a_i = [mi-r]_q$. From \eqref{eqn:a_eulerian_rec} we obtain the recursion
\begin{subequations} \label{eqn:q_r_W_eulerian_rec}
	\begin{align}
		A_q^{m,r}(n,k) &= 0, \qquad \text{for $k<0$ or $k>n$ }, \\
		A_q^{m,r}(0,0) &= 1, \\
		\intertext{and, for $k \geq 0$,}
		A_q^{m,r}(n+1,k) &=[m(n-k+2)-r]_q A_q^{m,r}(n,k-1) \notag \\
		&+ q^{m(n+1)-mk-r} [mk+r]_q A_q^{m,r}(n,k),
	\end{align}
\end{subequations}
For $z = [mx-r]_q$ the Worpitzky identity \eqref{eqn:a_eulerian_con} reduces to
\begin{equation}
	[mx-r]_q^n = \sum_{k=0}^{n} A_q^{m,r}(n,k) \eBinom{x+k-1}{n}{q^m}
\end{equation}
and from \eqref{eqn:A_explicit} we obtain the explicit formula
\begin{equation}
	A_q^{m,r}(n,k) = \sum_{j=0}^{k} (-1)^j q^{m\binom{n-j+1}{2}-n(m(k-j)+r)} \eBinom{n+1}{j}{q^m} [m(k-j)+r]_q^n.
\end{equation}
For $(m,r)=(1,0)$ this reduces to Carlitz' $q$-Eulerian numbers \eqref{eqn:q_eulerian_con}. The case $(m,r)=(m,1)$ was studied in \cite{CM} as a $q$-Eulerian number analogue for the wreath product $C_r \wr \mathfrak{S}_n$. For the case $(m,r)=(2,1)$ we get the $q$-Eulerian numbers of type $B$ \cite{CG}. 

\subsection{Elliptic Eulerian numbers}
One of the main motivations of this article was to introduce an elliptic analogue of Carlitz' $q$-Eulerian numbers. We will obtain such an analogue by specializing $A_\va(n,k)$ in the obvious way. 
We define \defn{elliptic Eulerian numbers} $A_{a,b;q,p}(n,k)$ by specializing $a_i=[i]_{a,b;q,p}$ in $A_\va(n,k)$ to obtain the recurrence relation
\begin{subequations} \label{eqn:ell_eulerian_rec}
	\begin{align}
		A_{a,b;q,p}(n,k) &= 0, \qquad \text{for $k<0$ or $k>n$}, \\
		A_{a,b;q,p}(0,0) &= 1, \\
		\intertext{and, for $k \geq 0$,}
		A_{a,b;q,p}(n+1,k) &=[n-k+2]_{a,b;q,p} A_\va(n,k-1) \nonumber \\
		&\quad + W_{a,b;q,p}(-1) [k]_{1/a,b/a;q,p} P_{a,b;q,p}(n,k) A_\va(n,k),
	\end{align}
\end{subequations}
where
\[
	P_{a,b;q,p}(n,k)=W_{aq^{-2k},bq^{-k};q,p}(n+1)\prod_{i=1}^{n+1} \frac{[n-i+2]_{aq^{2(i-k)},bq^{i-k};q,p}}{[n-i+2]_{aq^{2(i-k-1),bq^{i-k-1};q,p}}},
\]
where we used \eqref{eqn:weightshift} and \eqref{eqn:minuselliptic}. 
For the elliptic Eulerian numbers the generalized Worpitzky identity 
\begin{equation}
	\elliptic{z}^n = \sum_{k=0}^{n} A_\va(n,k) \prod_{i=1}^{n} \frac{[z-i+k]_{aq^{2(i-k)},bq^{i-k};q,p}}{[n-i+1]_{aq^{2(i-k)},bq^{i-k};q,p}}
\end{equation}
holds and from Theorem~\ref{thm:aeulerian_sum} we obtain the explicit sum-formula
\begin{align}
	A_{a,b;q,p} &= \sum_{j=0}^{k} \Big( \prod_{\substack{i=0\\ i\neq k-j}}^{n} \frac{[n-i+1]_{aq^{2(i-k)},bq^{i-k};q,p}}{[k-j-i]_{aq^{2(i-k)},bq^{i-k};q,p}} \Big) \elliptic{-j}^n\\
	&=(-1)^n W_{a,b;q,p}(-1)^n \sum_{j=0}^{k} \Big( \prod_{\substack{i=0\\ i\neq j}}^{n} \frac{[n-i+1]_{aq^{2(i-k)},bq^{i-k};q,p}}{[j-i]_{aq^{2(i-k)},bq^{i-k};q,p}} \Big) \elliptic{k-j}^n.
\end{align}
An \defn{elliptic analogue of the $r$-Whitney Eulerian numbers}, $A_{a,b;q,p}^{m,r}(n,k)$, can be obtained by specializing $a_i=[mi-r]_{a,b;q,p}$ in $A_\va(n,k)$. We omit the details here.
\section{Acknowledgements}
This work is based on research which is a part of the author’s doctoral thesis at the University of Vienna \cite{JosefDiss}. The author is grateful to Ilse Fischer, Hans H\"ongesberg, Michael Schlosser and Meesue Yoo for helpful and inspiring discussions.
\bibliographystyle{plain}
\bibliography{bibliography}

\end{document}